\documentclass[12pt]{article}
\usepackage{amsmath}
\usepackage{amsfonts}
\usepackage{amsthm}
\usepackage{amssymb}
\usepackage[showonlyrefs]{mathtools}
\mathtoolsset{showonlyrefs}

\newtheorem{lettertheorem}{Theorem}

\newtheorem{theorem}{Theorem}[section]
\newtheorem{lemma}[theorem]{Lemma}

\theoremstyle{remark}
\theoremstyle{definition}
\newtheorem{definition}[theorem]{Definition}

\numberwithin{equation}{section}
\begin{document}
\vskip 0.4in
\title{\bfseries\scshape{On the iterations and the argument distribution of meromorphic functions }}
\author{\bfseries\scshape $Zheng ~Jianhua$ \thanks{E-mail address:
zheng-jh@mail.tsinghua.edu.cn}$and$ ~$Ding ~Jie$
\\
1. Department of Mathematical Science, Tsinghua University, \\
Beijing, 100084, China\\
2. School of Mathematics, Taiyuan University of Technology, \\
Taiyuan,
Shanxi, 030024, P. R. China\\
}

\date{}
\maketitle \thispagestyle{empty} \setcounter{page}{1}

\begin{abstract} \noindent
This paper consists of tow parts. One is to study the existence of a
point $a$ in the intersection of Julia set and escaping set such
that $\arg z=\theta$ is a singular direction if $\theta$ is a limit
point of $\{\arg f^n(a)\}$ under some growth condition of a
meromorphic function. The other is to study the connection between
the Fatou set and singular direction. We prove that the absent of
singular direction deduces the non-existence of annuli in the Fatou
set.
\end{abstract}

\noindent {\bf 2010 Mathematics Subject Classification:} 37F10, 30D05.

\vspace{.08in} \noindent \textbf{Keywords}: Meromorphic function;
Fatou set; Julia set; Borel direction; Filling disk.

\section{Introduction and Main Results}

It is well-known that iteration and argument distribution of a
transcendental meromorphic function basically belong to different
topics in theory of meromorphic functions. The main objects studied
in the iteration theory of meromorphic functions are the Fatou set
and Julia set and those in the argument distribution are the
singular directions and filling disks. It would be interesting to
explore their connections. We try to do that in this paper. So let
us begin with the basic knowledge and notations from these two
topics.

For a transcendental meromorphic function $f$, denote by
    $$
    f^n:= \underbrace{f\circ \cdots \circ f}_n
    $$
the $n$-th iterate of $f$, for $n\in \mathbb{N}$. The Fatou set
$F(f)$ of $f$ is the set of points in $\mathbb{C}$ each of which has
a neighborhood where $\{f^n\}$ is well defined and forms a normal
family in the sense of Montel (or, equivalently, it is
equiv-continuous). The complement $J(f)$ of $F(f)$ with respect to
$\hat{\mathbb{C}}$ is called the Julia set of $f$. Both of the sets
are completely invariant under $f$, i.e., $z\in F(f)$ if and only if
$f(z)\in F(f)$. Let $U$ be a connected component of $F(f)$, then
$f^n(U)$ is contained in a component of $F(f)$, denoted by $U_n$. If
for some positive integer $p$, $f^p(U)\subseteq U_p=U$, then $U$ is
called a periodic Fatou component and such the smallest integer $p$
is the period of $U$; If for some $n>0$, $U_n$ is periodic, but $U$
is not periodic, then $U$ is called pre-periodic; If it is neither
periodic nor pre-periodic, that is, $U_n \not= U_m$ for all pair $n
\not= m$, then $U$ is called a wandering domain .

An introduction to the basic properties of these sets for rational
function can be found in \cite{beardon, milnor} and for
transcendental meromorphic function in the survey \cite{bergweiler}
or book \cite{zheng book}.

Let $f$ be a transcendental meromorphic function. Throughout this
article, denote by $M(r, f)$ and $L(r, f)$ the maximum modulus and
minimum modulus of $f$ on circle $\{z:|z|=r\}$, respectively. For a
domain $\Omega$, by $n(r,\Omega, f=a)$ we denote the number of roots
of $f=a$ counted with the multiplicities in $\Omega\cap\{z:|z|<r\}$,
and we briefly write $n(r, f=a)$ for $n(r,\Omega, f=a)$ if $\Omega$
is the complex plane and write $n(r,f)$ for $n(r,f=\infty)$. Set
$$N(r,\Omega,f=a)=\int_1^r\frac{n(t,\Omega,f=a)}{t}{\rm d}t.$$

Moreover, define
$$m(r,f)=\frac{1}{2\pi}\int_0^{2\pi}\log^+|f(re^{i\theta})|{\rm d}\theta,$$
$$N(r,f)=\int_0^r\frac{n(t,f)-n(0,f)}{t}{\rm d}t+n(0,f)\log r,$$
and $$T(r,f)=m(r,f)+N(r,f).$$ $T(r,f)$ is known as the Nevanlinna
characteristic of $f$. For a meromorphic function $f$, the growth
order $\rho(f)$ and lower order $\lambda(f)$ are respectively
defined as
    $$
    \rho(f)=\displaystyle \overline{\lim_{r \rightarrow \infty}}\frac{\log T(r,f)}{\log r}\quad\textrm{and}\quad
    \lambda(f)=\displaystyle \lim_{\overline{r \rightarrow \infty}}\frac{\log T(r,f)}{\log
    r}.
    $$
    The Nevanlinna deficiency $\delta(\infty,f)$ of $f$ at
    $\infty$ is
    $$\delta(\infty,f)=\lim\limits_{\overline{r\to\infty}}\frac{m(r,f)}{T(r,f)}$$
    and for $a\in \mathbb{C}$, $\delta(a,f)=\delta(\infty,1/(f-a)).$
    $a\in\hat{\mathbb{C}}$ is called deficient value of $f$ if
    $\delta(a,f)>0$. The deficiency and deficient value are
    main objects studied in the modulo distribution of the Nevanlinna
    theory. The singular direction is main objects studied in the argument
    distribution of the theory.

\begin{definition}
A direction $\text{arg}z=\theta$ is called a Borel direction of
order $\rho>0$ of a meromorphic function $f$, if for arbitrary
$\varepsilon>0$ and any $a\in \hat{\mathbb{C}}$, possibly except at
most two values of $a$, we have
\begin{equation}\label{1.3}\limsup_{r \rightarrow\infty}\frac{\log^+
N(r,Z_{\varepsilon}(\theta),f=a)}{\log r}\geq \rho,\end{equation}
where $Z_{\varepsilon}(\theta)=\{z: \theta-\varepsilon<\arg
z<\theta+\varepsilon\}$; A direction $\text{arg}z=\theta$ is called
a Julia direction of $f$ if (\ref{1.3}) is replaced by
$$\lim\limits_{r\to\infty}n(r,Z_{\varepsilon}(\theta),f=a)=\infty.$$
\end{definition}

It is easily seen that a Borel direction of positive order must be a
Julia direction.

This paper consists of two parts: one is to study the connection
between singular directions and Julia sets; the other is to study
the connection between singular directions and Fatou sets.

Now let us go to our first main purpose to consider connections
between the Julia set and singular direction. For a meromorphic
function $f$, the escaping set $I(f)$ is defined by
$$I(f)=\{z:\ f^n(z)\to\infty (n\to\infty)\}.$$
The escaping set is first introduced and investigated by Eremenko
\cite{eremenko} for transcendental entire functions and by Domiguze
\cite{Dominguez} for meromorphic functions. Many important dynamical
behaviors of it have been revealed in the references, for example,
see \cite{RS,RS1, R07, R06}.

The following result was proved in Qiao \cite{qiao}.

\begin{lettertheorem}\label{thmA}
Let $f$ be a transcendental meromorphic function of lower order
$\lambda\in (0,\infty)$. If
\begin{equation}\label{1.4}\mathcal{K}:=\overline{\lim\limits_{r\rightarrow \infty}}\frac{\log
T(2r,f)}{\log T(r,f)}<\infty,\end{equation} then there exists a
point $a\in I(f)\cap J(f)$ such that for each limit point $\theta$
of $\{\arg f^n(a)\}$, $\arg z=\theta$ is a Borel direction of order
at least $\lambda$.
\end{lettertheorem}

The quantity defined by the upper limit in (\ref{1.4}) is not larger
than $\frac{\rho(f)}{\lambda(f)}.$ Therefore, if
$0<\lambda(f)\leq\rho(f)<+\infty$, then (\ref{1.4}) holds. But the
condition (\ref{1.4}) does not exclude $\rho(f)=\infty$. In fact,
under (\ref{1.4}), we have
$$\overline{\lim\limits_{r\to\infty}}\frac{\log \log T(r,f)}{\log
r}\leq\frac{\log \mathcal{K}}{\log 2}<\infty.$$

We will prove that the condition (\ref{1.4}) is basically not
necessary. Closely related to the singular directions is the filling
disks.

\begin{definition}\ A disk $D:\ |z-z_0|<\varepsilon |z_0|$ is
called a filling disk of $f$ with index $m$, if $f$ takes all values
at least $m$ times on $D$ possibly except those values contained in
two spherical disks with radius at most $e^{-m}$.
\end{definition}

Henceforth, by $\chi(a,b)$ we denote the spherical distance between
$a$ and $b$. It is clear that a sequence of filling disks
$\Gamma_n:|z-z_n|<\varepsilon_n|z_n|$ with index $m_n$ determine
singular directions $\arg z=\theta$ where $\theta$ is a limit point
of $\{\arg z_n\}$, if $z_n\to\infty$, $m_n\to\infty$ and
$\varepsilon_n\to 0$ as $n\to\infty$; the order of the singular
direction is between $\lim\limits_{\overline{n\to\infty}}\frac{\log
m_n}{\log |z_n|}$ and $\overline{\lim\limits_{n\to\infty}}\frac{\log
m_n}{\log |z_n|}.$

\begin{theorem}\label{thm1.5} Let $f$ be a transcendental meromorphic function
with
$$\lim\limits_{r\to\infty}\frac{T(r,f)}{(\log r)^5}=\infty.$$
Then there exists a point $a\in I(f)\cap J(f)$ such that every
$f^n(a)$ lies in a filling disk
$\Gamma_n:|z-z_n|<\frac{8\pi}{\log\log R_n}|z_n|$ of $f$ with index
$m_n=c^*\frac{T(R_n,f)}{(\log\log R_n)^2(\log R_n)^3}$, where
$R_{n+1}>R^3_n$, $3R_n\geq |z_n|\to\infty (n\to\infty)$ and $c^*$ is
an absolute constant.
\end{theorem}

If $f$ has the lower order $\lambda>0,$ it is easy  from Theorem
\ref{thm1.5} to obtain the result in Theorem A without the
assumption (\ref{1.4}).

\begin{theorem}\label{thm1.4+}\ Let $f$ be a transcendental meromorphic function
with $0<\lambda(f)\leq\rho(f)<+\infty$. Let $E$ be the set of
$\theta\in[0,2\pi)$ such that $\arg z=\theta$ is a Borel direction
of $f$. Then there exists a point $a\in I(f)\cap J(f)$ such that the
set of the limit points of $\{\arg f^n(a)\}$ is $E$.
\end{theorem}

In \cite{RipponStallard}, Rippon and Stallard defined the slow
escaping points and proved their existence for a transcendental
meromorphic function.

\begin{theorem} \label{th2}
Let $f$ be a transcendental meromorphic function satisfying for a
$0<c<1$ and all sufficiently large $r$,
\begin{equation}\label{equ1.7}T(er,f)>\left(1+\frac{1}{(\log
r)^c}\right)T(r,f).\end{equation} Then for any increasing positive
sequence $\{a_n\}$ tending to $\infty$, there exists a point $a\in
I(f)\cap J(f)$ such that $|f^n(a)|\leq a_n$ and for each limit point
$\theta$ of $\{\arg f^n(a)\}$, $\arg z=\theta$ is a Borel direction
of order $\rho(f)$.
\end{theorem}

As we know that $T(r,f)$ is nondecreasing and logarithmic convex,
for large $r$ we always have
$$T(er,f)\geq\left(1+\frac{1}{\log r}\right)T(r,f).$$

In terms of the existence of filling disks, we can prove the
following.

\begin{theorem}\label{thm1.8} Let $f$ be a transcendental entire function. If
for $0<c<1$, (\ref{equ1.7}) holds, then $f$ has no multiply
connected Fatou components.
\end{theorem}

Theorem \ref{thm1.8} is an improvement of Corollary 5 in
\cite{zheng} and Theorem 1.3 in \cite{zheng 16} where (\ref{equ1.7})
is replaced by the inequality $T(er,f)>dT(r,f)$ with some $d>1$.

The fast escaping set of a transcendental entire function is
introduced in \cite{BH}. It is natural to ask whether there exists a
fast escaping point whose orbit go along singular directions. In
fact, for an entire function with regular growth $\log M(er,f)>d\log
M(r,f)$ and $d>1$, such a fast escaping point exists. For example,
we consider the exponential function $\lambda e^z$ with
$0<\lambda<1/e$. Its Julia set consists of uncountably many pairwise
disjoint simple $C^\infty$-smooth curves tending to $\infty$, called
hairs (which was proved by Devaney and Krych\cite{DK} and Viana
\cite{Vi}), and all of points on the hairs possible except their
finite endpoints are the fast escaping points (which was proved by
Devaney and Tangerman \cite{DT} and Rempe, Rippon and Stallard
\cite{RRS}). Therefore, a slow escaping point must be a finite
endpoint of some hair and the other points of the hair are the fast
escaping points and go to $\infty$ under the iterates far away from
the singular directions. $f(z)=\lambda e^z$ has only two singular
directions: one is the positive imaginary axis and the other is the
negative imaginary axis. For $-\frac{\pi}{2}+\varepsilon<\arg
z<\frac{\pi}{2}-\varepsilon$ and $|z|$ large,
$$|f(z)|\geq \lambda e^{|z|\sin\varepsilon}=\lambda^{1-\sin\varepsilon}M(|z|,f)^{\sin\varepsilon}>
M(|z|,f)^{\frac{1}{2}\sin\varepsilon}.$$ Therefore, for a point
$a\in I(f)$, if no limit points of $\{\arg f^n(a)\}$ are
$\pm\frac{\pi}{2}$, in view of a result in \cite{RipponStallard},
$a$ must be in the fast escaping set of $f$. We can prove that there
exist the finite endpoints of hairs which are fast escaping to
$\infty$ along the positive imaginary axis, i.e., the argument of
the iterate points tends to $\frac{\pi}{2}$. The Eremenko point
under iteration does not go along the singular directions, see
\cite{RS19} and the maximally fast escaping points introduced by
Sixsmith \cite{Six} go far away from the singular directions.

By $A(r,R)$ we denote the annulus $\{z:\ r<|z|<R\}$ and by $B(0,r)$
the disk $\{z:\ |z|<r\}$. In \cite{baker, zheng, BRS} for
transcendental entire functions it was proved that every
multiply-connected Fatou component $U$ is wandering and for all
sufficiently large $n$, $f^n(U)$ contains a round annulus
$A(r_n,R_n)$ with $r_n\to\infty, R_n/r_n\to\infty$ as $n\to\infty$.
This result was extended in \cite{zheng 16} to transcendental
meromorphic functions which have few poles. However, for a general
meromorphic function, a multiply-connected Fatou component may not
be wandering. There exist meromorphic functions which have a
sequence of large annuli in a periodic domain.

The second of our main purposes is about non-existence of a round
annulus $A(r,R)$ with $r>R_0$ in the Fatou set of a meromorphic
function under the condition of argument distribution.

\begin{theorem}\label{thm1.2} Let $f$ be a transcendental meromorphic function
with $\infty$ as its Nevanlinna deficient value and the lower order
$\lambda=\infty$. If there is a direction $\arg z=\theta$ is not a
Borel direction of the infinite order, then there exists
$\varepsilon(r)\to 0^+$ and a $R_0>0$ such that $F(f)$ contains no
annulus $A(r,(1+\varepsilon(r))r)$ with $r>R_0$.
\end{theorem}

In fact, we will establish the following general result with Theorem
\ref{thm1.2} as a corollary.

\begin{theorem} \label{thm1.3}
Let $f$ be a transcendental meromorphic function with the lower
order $\lambda$ and
\begin{equation}\label{1}\lim\limits_{\overline{r\to\infty}}\frac{\log
M(r,f)}{T(r,f)}>0.\end{equation} Assume that there are two distinct
values $a$ and $b$ and an angle
$\Omega=\Omega(\alpha,\beta)=\{z:\alpha<\arg z<\beta\}$ such that
$\lambda>\frac{\pi}{\beta-\alpha}=\omega$ and
\begin{equation}\label{2}\overline{\lim\limits_{r\to\infty}}\frac{\log
(N(r,\Omega,f=a)+N(r,\Omega,f=b))}{\log
T(r,f)}<1-\frac{\omega}{\lambda}.\end{equation} Let $\phi(r)$ be a
positive function in $[1,\infty)$ such that $\phi(r)\to\infty$ and
$\phi(r)/\log T(r,f)\to 0$ as $r\to\infty$ and $\phi(r)r/\log
T(r,f)>2\inf\limits_{z\in J(f)}|z|$ for $r\geq 1$.

Then there exists a $R_0>0$ such that the Fatou set $F(f)$ contains
no round annuli with the form $\{r<|z|<R\}$ for $r>R_0$ and
$R>(1+\phi(r)/\log T(r,f)) r$.
\end{theorem}

Furthermore, under the assumptions of Theorem \ref{thm1.3}, in view
of a result of \cite{Zheng3}, for any compact subset $W$ of a Fatou
component with $f^n|_W\to\infty (n\to\infty)$, there exists a
$M(W)>0$ such that for any pair $z$ and $z'$ in $W$,
$$M^{-1}(W)|f^n(z)|\leq |f^n(z')|\leq M(W)|f^n(z)|,\ \forall\ n\in \mathbb{N}.$$

If $\infty$ is a Nevanlinna deficient value of $f$, i.e.,
$\delta(\infty,f)=\lim\limits_{\overline{r\to\infty}}\frac{m(r,f)}{T(r,f)}>0$,
then obviously, (\ref{1}) holds. The quantity defined by the lower
limit in (\ref{1}) was introduced and studied in \cite{Petrenko}. It
is easy to see that (\ref{2}) follows from
\begin{equation}\label{equ1.4}\overline{\lim\limits_{r\to\infty}}\frac{\log
(N(r,\Omega,f=a)+N(r,\Omega,f=b))}{\log
r}<\lambda-\omega.\end{equation} This means that the convergence
exponents of $a$-points and $b$-points are smaller than
$\lambda-\omega$. But (\ref{2}) does not exclude the possibility of
that the convergence exponent equals to $\lambda-\omega$, i.g., if
$\lambda=\infty$, the exponent allows to be $\infty$. Therefore,
Theorem \ref{thm1.3} is a generalization of Corollary 4 in
\cite{zheng}. And the condition (\ref{equ1.4}) has something to do
with the singular directions of a meromorphic function. If $\arg
z=\theta$ is not a Borel direction of $f$ with the order
$\lambda-\omega$, then there exist an angle $\Omega$ containing the
ray and two values $a$ and $b$ such that (\ref{equ1.4}) holds.
Therefore Theorem \ref{thm1.2} follows from Theorem \ref{thm1.3}. In
view of a result of Valiron (see Theorem 2.7.5 in \cite{zheng book
2}), if $f$ has no Borel directions of order $\lambda-\omega$ in
$\Omega$, then (\ref{equ1.4}) holds for some $a$ and $b$.

From the proof of Case B in the proof of Theorem \ref{thm1.3}, we
can establish the following

\begin{theorem}\ Let $f$ be a meromorphic function with (\ref{1}) without (\ref{2}) and
for some $a\in\mathbb{C}$, $\delta(a,f)>0$. Then the results in
Theorem \ref{thm1.3} holds.
\end{theorem}

Next let us give an example to show that the condition (\ref{1}) is
necessary.

\begin{theorem}\label{thm1.5+}\ For any given $\lambda>1$, there exists a meromorphic function of order and lower order
equal to $\lambda$ such that (\ref{2}) holds on the upper half plane
and lower half plane and its Julia set lies on the real axis and its
Fatou set contains a sequence of annuli $A(r_n,dr_n)$ with $d>1$ and
$r_n\to\infty$ as $n\to\infty$.
\end{theorem}

\section{Proofs of Theorems \ref{thm1.5}, \ref{thm1.4+}, \ref{th2} and \ref{thm1.8}}

\subsection{Some Lemmas}

The following result is natural; see Lemma 1 in
\cite{RipponStallard}.

\begin{lemma}\label{lem4.1}\ Let $f$ be a meromorphic function and let $\{E_n\}_{n=0}^\infty$ be a sequence of compact
sets in $\mathbb{C}$. If $$E_{n+1}\subset f(E_n),\ \text{for}\ n\geq
0,$$ then there exists a $\xi$ such that $f^n(\xi)\in E_n$, for
$n\geq 0$. If $E_n\cap J(f)\not=\emptyset$, for $n\geq 0$, then
$\xi$ can be chosen to be in $J(f)$.
\end{lemma}

The following result extracts from the proof of Lemma 6 and Lemma 7
in \cite{RipponStallard}

\begin{lemma}\label{lem4.2}\ Let $f$ be a meromorphic function. If there
exist two sequences $\{B_m\}_{m=0}^\infty$ and
$\{V_m\}_{m=0}^\infty$ of compact sets with ${\rm
dist}(0,B_m)\to\infty$ and ${\rm dist}(0,V_m)\to\infty$ as
$m\to\infty$ and a strictly increasing sequence of positive integers
$\{m(k)\}$ such that
\begin{equation}\label{eq4.1}B_{m+1}\subseteq f(B_m),\
B_{m(k)-p(k)}\subseteq f(V_k),\ V_k\subseteq
f(B_{m(k)}),\end{equation} where $0\leq p(k)\leq M$ for a fixed
integer $M>0$, then for any increasing sequence of positive numbers
$\{a_n\}$ with $a_n\to\infty (n\to\infty)$, there exists a $\zeta\in
I(f)$ such that for all sufficiently large $n$, $|f^n(\zeta)|\leq
a_n$; If, in addition, $B_m\cap J(f)\not=\emptyset$, $\forall\ m\geq
0$, then we can require $\zeta\in J(f)$.
\end{lemma}

\begin{proof}\ We choose a subsequence $\{a_{n(m)}\}$ of $\{a_n\}$
such that
$$B_{p}\subset B(0,a_{n(m)}),\ \text{for}\ 0\leq p\leq m,\ \text{and}\ V_m\subset B(0,a_{n(m)}).$$
Inductively we construct a sequence $\{s(k)\}$ of positive integers
which are used to control the speed of iterates of $f$ on $B_m$. Set
$d(k)=s(k)p(k)+2s(k)$ and
$q(k)=d(0)+d(1)+d(2)+...+d(k)=q(k-1)+d(k)$.

Define $s(0)=0$. Suppose that we have had $s(k-1)$ and so $d(k-1)$
and $q(k-1)$ are fixed. Take $s(k)$ such that
$$m(k)+q(k)>n(m(k+1)).$$
Let us construct a sequence $\{E_n\}$ of compact sets as follows:
$$\begin{array}{l}
E_0=B_0,\ E_1=B_1,\ ...\ , E_{m(1)}=B_{m(1)}\\
E_{m(1)+1}=V_1\\
E_{m(1)+2}=B_{m(1)-p(1)},\ ...\ ,\ E_{m(1)+p(1)+2}=B_{m(1)}\\
......\\
E_{m(1)+jp(1)+2j+1}=V_1\\
E_{m(1)+jp(1)+2j+2}=B_{m(1)-p(1)},\ ...\ ,\ E_{m(1)+(j+1)p(1)+2j+2}=B_{m(1)},\\
E_{m(1)+(j+1)p(1)+2(j+1)+1}=V_1,\ 0\leq j\leq s(1),\\
E_{m(1)+q(1)}=B_{m(1)},\ q(1)=d(1)=s(1)p(1)+2s(1),\\
E_{m(1)+1+q(1)}=B_{m(1)+1},\ ...\ ,\ E_{m(2)+q(1)}=B_{m(2)},\\
E_{m(2)+q(1)+1}=V_2,\\
E_{m(2)+q(1)+2}=B_{m(2)-p(2)},\ ...\ ,\ E_{m(2)+q(1)+p(2)+2}=B_{m(2)}\\
 ......
\end{array}$$
that is to say, for $k\geq 0$,
$$E_n=\left\{\begin{array}{ll}B_{n-q(k)}, & m(k)+q(k)\leq n\leq m(k+1)+q(k);\\
V_{k+1}, & n=m(k+1)+q(k)+jp(k+1)+2j+1;\\
\ & m(k+1)+q(k)+jp(k+1)+2j+2\\
B_{n-q(k)-(j+1)p(k+1)-2j-2}, & \leq n<m(k+1)+q(k)\\
\ &\ +(j+1)p(k+1)+2j+1,\\
\ & 0\leq j\leq s(k+1)\end{array}\right.$$ Then it is easy to see
that $E_{n+1}\subseteq f(E_n)$. Since $m(k)-p(k)\to\infty\
(k\to\infty)$ and ${\rm dist}(0,E_n)\to \infty\ (n\to\infty)$, in
view of Lemma \ref{lem4.1}, there exists a point $\zeta\in B_0\cap
I(f)$ such that $f^n(\zeta)\in E_n$.

For $m(k)+q(k)\leq n\leq m(k+1)+q(k)$, we have
$$E_n=B_{n-q(k)}\subset B(0,a_{n(m(k+1))})\subset B(0,a_n)$$
by noting that $n(m(k+1))<m(k)+q(k)\leq n$. When
$n=m(k+1)+q(k)+jp(k+1)+2j+1$, we have $n(k+1)<n(m(k+1))<m(k)+q(k)<n$
and so
$$E_n=V_{k+1}\subset B(0,a_{n(k+1)})\subset B(0,a_n).$$
For $m(k+1)+q(k)+jp(k+1)+2j+2\leq n<m(k+1)+q(k)+(j+1)p(k+1)+2j+1$,
that is, $m(k+1)-p(k+1)\leq n-q(k)-(j+1)p(k+1)-2j-2\leq m(k+1)$, we
have
$$E_n=B_{n-q(k)-(j+1)p(k+1)-2j-2}\subset B(0,a_{n(m(k+1))})\subset B(0,a_n).$$

We have proved that for all $n\geq m(1)+q(1)$, $E_n\subset
B(0,a_n).$
\end{proof}

The result in Lemma \ref{lem4.2} also holds if the condition
(\ref{eq4.1}) is replaced by
$$B_{m+1}\subseteq f(B_m),\ B_{m(k)-p(k)}\subseteq f(B_{m(k)}).$$

For a hyperbolic domain $U$, by $\lambda_U(z)$ we denote the
hyperbolic density of $U$ at $z\in U$ and by $d_U(z_1,z_2)$ the
hyperbolic distance between $z_1$ and $z_2$ in $U$.

\begin{lemma}(\cite{zheng 16}, Theorem 2.2)\label{lem3.2} \ Let $f$ be analytic on a hyperbolic domain $U$ with $0\not\in
f(U)$. If there exist two distinct points $z_1$ and $z_2$ in $U$
such that $|f(z_1)|> e^{\kappa\delta}|f(z_2)|$, where
$\delta=d_U(z_1,z_2)$ and
$\kappa=\Gamma(\frac{1}{4})^4/(4\pi)^2=4.3768796...$, then there
exists a point $\hat{z}\in U$ such that $|f(z_2)|\leq
|f(\hat{z})|\leq |f(z_1)|$ and
\begin{equation}\label{2.2}f(U)\supset
A\left(e^\kappa\left(\frac{|f(z_2)|}{|f(z_1)|}\right)^{1/\delta}|f(\hat{z})|,\
e^{-\kappa}\left(\frac{|f(z_1)|}{|f(z_2)|}\right)^{1/\delta}|f(\hat{z})|\right);\end{equation}
If $|f(z_1)|\geq
\exp\left(\frac{\kappa\delta}{1-\delta}\right)|f(z_2)|$ and
$0<\delta<1$, then
\begin{equation}\label{2.8}
f(U)\supset A(|f(z_2)|,|f(z_1)|).
\end{equation}
In particular, for $\delta\leq\frac{1}{6}$ and $|f(z_1)|\geq
e|f(z_2)|$, we have (\ref{2.8}).
\end{lemma}

We need a result on the existence of filling disks.

\begin{lemma}\label{lem4.3} (cf. \cite{yang book}, Lemma 3.4)
Let $f$ be a transcendental meromorphic function. Given a $q>1$ and
$R$ satisfying $$T(R,f)\geq \max\left\{240, \frac{240 \log(2R)}{\log
k}, 12T(r,f), \frac{12T(kr,f)}{\log k} \log\frac{2R}{r}\right\}$$
for some $k>1$, then there exists a point $z_j$ with $r <|z_j|< 2R$
such that the disk $$\Gamma: |z-z_j|<\frac{4\pi}{q}|z_j|$$ is a
filling disk with index
\begin{equation}\label{equ4.2}
m=c^*\frac{T(R,f)}{q^2(\log \frac{r}{R})^2},\end{equation} where
$c^*
> 0$ is an absolute constant.
\end{lemma}

There is a $r_0\geq 0$ such that $T(r,f)\equiv$ constant for
$r\in[0,r_0)$ and $T(r,f)$ is strictly increasing in
$[r_0,+\infty)$. Therefore, $T(r,f)$ is invertable in
$[r_0,+\infty)$. We denote the inverse of $T(r,f)$ in $[r_0,
+\infty)$ by $T^{-1}(r,f)$ with $r\geq T(r_0,f)$. In other words, we
write Lemma \ref{lem4.3} as follows.

\begin{lemma}\label{lem4.4}
Let $f$ be a transcendental meromorphic function. Then there exists
a $R_0>0$ such that for $R>R_0$ and $q>R_0$, the annulus $A(r, 3R)$
contains a filling disk $$\Gamma: |z-a|<\frac{4\pi}{q}|a|$$ of $f$
with index $m$ given in (\ref{equ4.2}) where
\begin{equation}\label{equ4.3}r=\frac{1}{2e}T^{-1}\left(\frac{T(R,f)}{12\log R},f\right)\end{equation} and
$T^{-1}(*,f)$ is the inverse of $T(r,f)$.
\end{lemma}

We can obtain the following by making a modification of the proof of
Rauch Theorem 3.11 in \cite{yang book} which asserts that a Borel
direction confirms the existence of a sequence of filling disks.

\begin{lemma}\label{lem2.6+}\ Let $f$ be a transcendental meromorphic function and
$\arg z=\theta$ be a Borel direction of $f$ of order $\mu>0$. Then
there exist a sequence of filling disks:
$$\Gamma_j:\ |z-z_j|<\varepsilon_j|z_j|,\ z_j=|z_j|e^{i\theta},\
j=1,2,\cdots,$$
$$\lim_{j\to\infty}|z_j|=+\infty,\
\lim_{j\to\infty}\varepsilon_j=0$$ with index
$m_j=|z_j|^{\mu-\delta_j}$ and $\lim\limits_{j\to\infty}\delta_j=0.$
\end{lemma}

Since $f$ is transcendental, it is clear that as $R\to\infty$, we
have $\frac{T(R,f)}{12\log R}\to\infty$ so that $r\to\infty$. Now we
state the Ahlfors-Shimizu characteristic of a meromorphic function;
see \cite{Hayman}. By $f^\#(z)$ we mean the sphere derivative of $f$
at $z$. For a closed domain $D$, define
$$\mathcal{A}(D,f)=\iint_D(f^\#(z))^2{\rm d}\sigma(z)$$
and write $\mathcal{A}(r,f)$ for $\mathcal{A}(\overline{B}(0,r),f)$.
The Ahlfors-Shimizu characteristic of $f$ is defined as
$$\mathcal{T}(r,f)=\int_0^r\frac{\mathcal{A}(t,f)}{t}{\rm d}t.$$
Then we have
$$|T(r,f)-\mathcal{T}(r,f)-\log^+|f(0)||\leq \frac{1}{2}\log 2.$$

To confirm the existence of filling disk in view of Lemma
\ref{lem4.3} and Lemma \ref{lem4.4} in our proofs of theorems, we
need to give another formula of (\ref{equ1.7}).

\begin{lemma}\ Let $f$ be a transcendental meromorphic function.
Assume that for a $0<c<1$ and all sufficiently large $r$,
(\ref{equ1.7}) holds. Then for large $r$, we have
\begin{equation} \label{th2 assumption}
\frac{T(r^2,f)}{\log r^2}>12T({\rm e}r,f).
\end{equation}
\end{lemma}

\begin{proof} Under (\ref{equ1.7}), for $\sigma>0$ we have
\begin{eqnarray}\label{equ1.8}T(r^{1+\sigma},f)&=&T(e^{\sigma\log r}r,f)>\left(1+\frac{1}{((1+\sigma)\log
r-1)^c}\right)T(e^{\sigma\log r-1}r,f)\nonumber\\
&>&\prod_{k=1}^{[\sigma\log r-2]}\left(1+\frac{1}{((1+\sigma)\log
r-k)^c}\right)\cdot T(e^2r,f)\nonumber\\
&>&\left(1+\frac{1}{((1+\sigma)\log r)^c}\right)^{[\sigma\log r-2]}T(e^2r,f)\nonumber\\
&>&\exp\frac{[\sigma\log r-2]}{((1+\sigma)\log r)^c+1}\cdot T(e^2r,f)\nonumber\\
&>&12(1+\sigma)(\log r) T(e^2r,f).\end{eqnarray} That is to say, we
obtain (\ref{th2 assumption}) when $\sigma$ is chosen to be $1$.
\end{proof}
Let us make a remark on (\ref{th2 assumption}). For arbitrarily
large integer $N>0$ and sufficiently large $r$, we have from
(\ref{th2 assumption}) that
$$T(r,f)\geq \frac{6^N}{2^{2^N-1}}(\log r)^NT(r^{1/2^N},f).$$
Therefore
\begin{equation}\label{equ1.6}\lim_{\overline{r\to\infty}}\frac{T(r,f)}{(\log
r)^N}=\infty.\end{equation}

The final lemma comes from the calculus.

\begin{lemma}\ \label{lem4.5} For $R>0$, the function
$\sqrt{\frac{1+x^2}{1+(x-R)^2}}$ is increasing in $[0,\
\frac{1}{2}(R+\sqrt{4+R^2})]$ and decreasing in
$[\frac{1}{2}(R+\sqrt{4+R^2}),\infty)$ and
$$\sqrt{\frac{1+x^2}{1+(x-R)^2}}\leq \frac{1}{2}(R+\sqrt{4+R^2}),
\forall\ x>0.$$
\end{lemma}

\subsection{Proof of Theorem \ref{thm1.5}}

Let $R_0$ be as in Lemma \ref{lem4.4}. Take a $R_1$ with $q=\log\log
R_1>R_0$, in view of Lemma \ref{lem4.4}, there exists a filling disk
$\Gamma_1$ in $A(r_1,3R_1)$ with index
$$m_1=c^*\frac{T(R_1,f)}{(\log\log R_1)^2(\log R_1)^2}>6\log R_1$$ and $r_1$
defined by (\ref{equ4.3}) with $R=R_1$. Noting that
$\chi(z,\infty)>\frac{3}{2}e^{-m_1}$ on $|z|=\frac{1}{2}e^{m_1}$ and
the spherical diameter of the circle $|z|=\frac{1}{2}e^{m_1}$ is
larger than $2e^{-m_1}$, according to the definition of filling
disks, there exists a point $z_1\in\Gamma_1$ such that
$|f(z_1)|=\frac{1}{2}e^{m_1}$. Set
$$W_1=\{z:\ 2<|z|<\frac{1}{128}|f(z_1)|\}.$$
We need to treat three cases.

{\bf Case A.}\ Assume that there exists a point $w_0\in W_1$ such
that $w_0\not\in f(5\Gamma_1)$ and $f$ is analytic in $5\Gamma_1$.
Here and henceforth, for a disk $\Gamma=B(a,r)$, we define
$5\Gamma=B(a,5r)$. Since ${\rm diam}_{\chi}(B(w_0,2))>5e^{-m_1}$,
there exists a point $z_0\in\Gamma_1$ such that $|f(z_0)-w_0|\leq
2$. Set $g(z)=f(z+w_0)-w_0$. Then $0\not\in g(H_1)$ with
$H_1=5\Gamma_1-w_0$. In view of Lemma \ref{lem3.2}, by noting that
$d_{H_1}(z_0-w_0,z_1-w_0)=d_{5\Gamma_1}(z_0,z_1)\leq\log\frac{17}{7}<1$
and $|g(z_1-w_0)|\geq |f(z_1)|-|w_0|\geq\frac{1}{2}|f(z_1)|,$ we
have
$$g(H_1)\supset A(2,|g(z_1-w_0)|)\supset
A\left(2,\frac{1}{2}|f(z_1)|\right).$$ Set
$\hat{R}_2=\frac{1}{32}|f(z_1)|$. In view of Lemma \ref{lem4.4}, $g$
has a filling disk $\hat{\Gamma}_2$ in $A(\hat{r}_2,3\hat{R}_2)$
with index
$$\hat{m}_2=c^*\frac{T(\hat{R}_2,g)}{(\log\log \hat{R}_2)^2(\log \hat{R}_2)^2}$$
and $\hat{r}_2=\frac{1}{2e}T^{-1}\left(\frac{T(\hat{R}_2,g)}{12\log
\hat{R}_2},g\right)>32.$ Thus $5\hat{\Gamma}_2\subset g(H_1).$

By $\hat{\gamma}_{21}$ and $\hat{\gamma}_{22}$ we denote the two
exceptional disks of $g$ for the filling disk $\hat{\Gamma}_2$. Then
$\gamma_{21}=\hat{\gamma}_{21}+w_0$ and
$\gamma_{22}=\hat{\gamma}_{22}+w_0$ are the exceptional disks of $f$
for $\Gamma_2=\hat{\Gamma}_2+w_0$. We can assume without any loss of
generality that the center of $\gamma_{21}$ is a finite number and
$\infty$ is the center of $\gamma_{22}$. For any two points $z_1$
and $z_2$ in $\gamma_{21}$, set $\hat{z}_1=z_1-w_0$ and
$\hat{z}_2=z_2-w_0$ in $\hat{\gamma}_{21}$. Then
$$\chi(z_1,z_2)
=\frac{\sqrt{1+|\hat{z}_1|^2}\sqrt{1+|\hat{z}_2|^2}}{\sqrt{1+|z_1|^2}\sqrt{1+|z_2|^2}}\chi(\hat{z}_1,\hat{z}_2)$$
$$\leq
\sqrt{\frac{1+|\hat{z}_1|^2}{1+(|\hat{z}_1|-|w_0|)^2}}\sqrt{\frac{1+|\hat{z}_2|^2}{1+(|\hat{z}_2|-|w_0|)^2}}
\chi(\hat{z}_1,\hat{z}_2)$$
$$\leq\frac{1}{4}(|w_0|+\sqrt{4+|w_0|^2})^2\chi(\hat{z}_1,\hat{z}_2)<3R_2^2\chi(\hat{z}_1,\hat{z}_2).$$
And for a point
$\hat{z}\in\hat{\gamma}_{22}$, by noting that
$\chi(\hat{z},\infty)<e^{-\hat{m}_2}$, we have
$|\hat{z}|>e^{\hat{m}_2}-1>R^3_2>2|w_0|$ and so for
$z=\hat{z}+w_0\in\gamma_{22}$,
$$\chi(z,\infty)=\frac{\sqrt{1+|\hat{z}|^2}}{\sqrt{1+|z|^2}}\chi(\hat{z},\infty)\leq
\sqrt{\frac{1+|\hat{z}|^2}{1+(|\hat{z}|-|w_0|)^2}}\chi(\hat{z},\infty)$$
$$<\frac{|\hat{z}|}{|\hat{z}|-|w_0|}\chi(\hat{z},\infty)<2\chi(\hat{z},\infty).$$

Set \begin{equation}\label{equ4.4}m_2=c^*\frac{T(R_2,f)}{5(\log\log
R_2)^2(\log R_2)^3},\end{equation} with $R_2=\frac{1}{4}\hat{R}_2.$
Since $B(0,R_2)\subset B(0,\frac{1}{2}\hat{R}_2)+w_0$, we have
$$\mathcal{A}\left(\frac{1}{2}\hat{R}_2,f(z+w_0)\right)
=\mathcal{A}\left(B\left(0,\frac{1}{2}\hat{R}_2\right)+w_0,f(z)\right)$$
\begin{equation}\label{3.5}\geq \mathcal{A}(B(0,R_2),f(z))=\mathcal{A}(R_2,f).\end{equation} Therefore, we have
\begin{eqnarray*}T(\hat{R}_2,g)&=&T(\hat{R}_2,f(z+w_0)-w_0)\\
&=&N(\hat{R}_2,f(z+w_0))+m(\hat{R}_2,f(z+w_0)-w_0)\\
&\geq& T(\hat{R}_2,f(z+w_0))-\log |w_0|-\log 2\\
&\geq&
T(\hat{R}_2,f(z+w_0))-T\left(\frac{1}{2}\hat{R}_2,f(z+w_0)\right)-\log
|w_0|-\log 2\\
&\geq&\mathcal{T}(\hat{R}_2,f(z+w_0))-\mathcal{T}\left(\frac{1}{2}\hat{R}_2,f(z+w_0)\right)-\log|w_0|-\frac{3}{2}\log
2\\
&=&
\int_{\frac{1}{2}\hat{R}_2}^{\hat{R}_2}\frac{\mathcal{A}(t,f(z+w_0))}{t}{\rm
d}t-\log|w_0|-\frac{3}{2}\log
2\\
&\geq&\mathcal{A}\left(\frac{1}{2}\hat{R}_2,f(z+w_0)\right)\log
2-\log|w_0|-\frac{3}{2}\log
2\\
&\geq& \mathcal{A}(R_2,f)\log 2-\log|w_0|-\frac{3}{2}\log 2\\
& \geq&\frac{\log 2}{\log
R_2}\int_{1}^{R_2}\frac{\mathcal{A}(t,f)}{t}{\rm
d}t-\log|w_0|-\frac{3}{2}\log
2\\
&=&\frac{\log 2}{\log R_2}T(R_2,f)+O(1)-\log|w_0|-\frac{3}{2}\log
2\\
&\geq &\frac{1}{2\log R_2}T(R_2,f).
\end{eqnarray*}
This implies that $\hat{m}_2>2m_2$ and so
$3R^2_2e^{-\hat{m}_2}<3R^2_2e^{-2{m}_2}<e^{-m_2}.$ Therefore, $f$
has the filling disk $\Gamma_2$ with index $m_2.$ Since
$5\hat{\Gamma}_2\subset g(H_1)$, we have $5\Gamma_2-w_0\subset
f(5\Gamma_1)-w_0$ and so $5\Gamma_2\subset f(5\Gamma_1).$

{\bf Case B.}\ Assume that $W_1\subset f(5\Gamma_1)$ and $f$ is
analytic in $5\Gamma_1$. In view of Lemma \ref{lem4.4}, we have a
filling disk $\Gamma_2$ of $f$ in $A(r_2,3R_2)\subset W_1$ with
index $m_2$, where $R_2=\frac{1}{640}|f(z_1)|$, and
$5\Gamma_2\subset f(5\Gamma_1)$.

{\bf Case C.}\ Assume that $5\Gamma_1$ contains a pole of $f$. Then
for some $\hat{r}_2>0$, $\{z:\ |z|>\hat{r}_2\}\subset f(5\Gamma_1)$.
Take a sufficiently large $R_2>R^3_1$ such that $r_2>5\hat{r}_2$. In
view of Lemma \ref{lem4.4}, there exists a filling disk $\Gamma_2$
of $f$ in $A(r_2,3R_2)$ with index $m_2$. Obviously,
$5\Gamma_2\subset A(r_2/5,15R_2)\subset f(5\Gamma_1).$

In one word, there exists a filling disk $\Gamma_2$ of $f$ with
index $m_2$ given in (\ref{equ4.4}) where $R_2>R^3_1$ and
$5\Gamma_2\subset f(5\Gamma_1)$. Proceeding step by step, we obtain
a sequence of filling disks $\{\Gamma_n\}$ of $f$ with index $m_n$
given by (\ref{equ4.4}) with $R_2$ replaced by $R_n$ and
$R_n>R_{n-1}^3\to\infty (n\to\infty)$ and $5\Gamma_n\subset
f(5\Gamma_{n-1})$. Since $f$ is meromorphic on the complex plane,
$f$ takes any value at most finitely many times on any bounded
subset of the complex plane and therefore, by $m_n\to\infty
(n\to\infty)$, we know that ${\rm dist}(\Gamma_n,0)\to\infty
(n\to\infty).$ It is easily seen that for every $n$, $\Gamma_n\cap
J(f)\not=\emptyset.$ In view of Lemma \ref{lem4.1}, there exists a
point $a\in I(f)\cap J(f)$ such that $f^n(a)\in 5\Gamma_n$. Of
course, $5\Gamma_n$ is also a filling disk of $f$ with index $m_n$.

We complete the proof of Theorem \ref{thm1.5}.

\subsection{Proof of Theorem \ref{thm1.4+}}

From the proof of Theorem \ref{thm1.5}, we have a sequence of
filling disks $B_n$ such that $B_{n+1}\subset f(B_n)$ centered at
$z_n\to\infty$ with index $m_n$ having the order in
$[\lambda,\rho]$. Therefore, the radius of exceptional disks is
$e^{-m_n}\to 0(n\to\infty)$.

Noting that $E$ is closed in $[0,2\pi]$, we can choose a sequence of
$\{\theta_p\}_{p=1}^N$ with $1\leq N\leq+\infty$ such that the
closure $\overline{\{\theta_p:\ p=1,2,\cdots N\}}=E$. In view of
Lemma \ref{lem2.6+}, for every $\theta_p$, there exist a sequence of
filling disks $\{A_{pj}\}_{j=1}^\infty$ with index $m_{pj}$ centered
at $z_{pj}=r_{pj}e^{i\theta_p}, j=1,2,\cdots$ with $r_{pj}\to\infty
(j\to\infty)$ and $m_{p(j+1)}>m_{pj}$. And we can require that
$r_{p(j+1)}>2r_{pj}$ and $r_{(p+1)1}>2r_{p1}\to+\infty
(p\to\infty)$. Since the exceptional disks of $A_{pj}$ has the
radius at most $e^{-m_{pj}}$, choosing the sufficiently large
$r_{p1}$ we have that for every $j$ there exists a $B_{n(pj)}$ with
$n(pj)\to\infty(p\to\infty)$ such that
$$B_{n(pj)}\subset f(A_{pj}).$$

Now let us write $A_{pj}, j=1,2,\cdots, p=1,2,\cdots$ in the
following order:
$$A_{11},\ A_{12},\ A_{21}, A_{13}, A_{22}, A_{31},\cdots.$$
Take a sufficiently large $n_0$ and then one of $A_{11},\ A_{12},\
A_{13}$ is in $f(B_{n_0})$, denoted by $C_{11}$. $f(C_{11})$
contains one $B_{n_{11}}$. Take $B_{n_{11}+1},...,B_{n_{11}+p_{11}}$
such that $f(B_{n_{11}+p_{11}})$ contains one of $A_{12},\ A_{13},
A_{14}$, denoted by $C_{12}$. Then $f(C_{12})$ contains one
$B_{n_2}$. Take $B_{n_{12}+1},...,B_{n_{12}+p_{12}}$ such that
$f(B_{n_{12}+p_{12}})$ contains one of $A_{21}, A_{22}, A_{23}$,
denoted by $C_{21}$. We go forever in this way to obtain a sequence
of filling disks:
\begin{eqnarray*}
&B_{n_0}\\
&C_{11},B_{n_{11}}, B_{n_{11}+1},\cdots,B_{n_{11}+p_{11}},\\
&\cdots\\
&C_{sk},B_{n_{sk}},B_{n_{sk}+1},\cdots,B_{n_{sk}+p_{sk}}\\
&\cdots,\end{eqnarray*} where $C_{sk}$ is one of $A_{sk},
A_{s(k+1)}$ and $A_{s(k+2})$. In view of Lemma \ref{lem4.1}, there
exists a point $a\in I(f)\cap J(f)$ such that $f^n(a)$ goes along
the sequence of above filling disks. Then $a$ satisfies our
requirement.

\subsection{Proof of Theorem \ref{th2}}

Under (\ref{equ1.7}), we have (\ref{th2 assumption}). Take $r_1$
sufficiently large and set $R_1=r_1^2$ such that
$$T(R_1,f)\geq \max\left\{240, \frac{240 \log(2R_1)}{\log
2}, 12T(er_1,f)\log (2r_1)\right\},$$ $q_1=\log r_1$ and, in view of
(\ref{equ1.6}),
$$c^*\frac{T(R_1,f)}{(\log r_1)^4}>2\log (1+2R_1),$$
where $c^*$ is the constant in Lemma \ref{lem4.3}. Applying Lemma
\ref{lem4.3}, then there exists a $z_1$ lying in annulus $\{z:
r_1<|z|<2R_1\}$ such that
$$\Gamma_1: |z-z_1|<\frac{4\pi}{q_1}|z_1|$$
is a filling disk of $f$ with index
$$n_1=c^*\frac{T(R_1,f)}{q_1^2(\log r_1)^2}=c^*\frac{T(R_1,f)}{(\log r_1)^4}.$$

Take $r_2,\ R_2=r_2^2$ and $q_2=\log r_2$ such that
$$r_2>\frac{2R_1+(1+2R_1)e^{-n_1}}{1-e^{-n_1}(1+2R_1)}.$$
This implies that $\chi(r_2,2R_1)>e^{-n_1}$, where $\chi(z,w)$
denotes the spherical distance of $z$ and $w$. There exists a $z_2$
lying in annulus $\{z: r_2<|z|<2R_2\}$ such that
$$\Gamma_2: |z-z_2|<\frac{4\pi}{q_2}|z_2|$$
is a filling disk of $f$ with index
$$n_2=c^*\frac{T(R_2,f)}{(\log r_2)^4}.$$ And the spherical distance of $\Gamma_1$ and $\Gamma_2$ is at
least $\chi(r_2,2R_1)>e^{-n_1}>e^{-n_2}.$

In view of the same method, we have $r_j, R_j, q_j, n_j$ and
$\Gamma_j\ (j=1,2,\cdots,5)$ such that the spherical distance of
$\Gamma_j$ and $\Gamma_i$ with $i\not=j$ is larger than $e^{-n_t}
(t=1,2,\cdots,5)$.

Take $B_1=\Gamma_1$. We can have $B_2=\Gamma_i$ for some $2\leq
i\leq 4$ with $f(B_1)\supset B_2$ and then $B_3=\Gamma_j$ for some
$j\in\{2,3,4,5\}\setminus\{i\}$ with $f(B_2)\supset B_3$. It is easy
to see that $f(B_3)\supset B_1, B_2$ or $B_3$. Starting from $B_3$,
we go on the same step to obtain $B_4$ and $B_5$ such that
$f(B_j)\supset B_{j+1}$ with $j=3,4$ and $f(B_5)\supset B_3, B_4$ or
$B_5$. Thus we obtain a sequence of disks $\{B_n\}$ such that
$f(B_n)\supset B_{n+1}$ and $f(B_{2n+1})\supset B_{2n-1}, B_{2n}$ or
$B_{2n+1}$.

We can take $r_n$ such that $\lim\limits_{n\to\infty}\frac{\log
T(R_n,f)}{\log R_n}=\rho(f)$. In view of Lemma \ref{lem4.2}, we
complete the proof of Theorem \ref{th2}.

\subsection{Proof of Theorem \ref{thm1.8}}

It follows from (\ref{equ1.8}) that
$$\frac{1}{2e}T^{-1}\left(\frac{T(r^{1+\sigma},f)}{12\log
r^{1+\sigma}},f\right)>r.$$ In view of Lemma \ref{lem4.4}, for any
$\sigma>0$ and all sufficiently large $r$, the annulus
$A(r^{1-\sigma},r)$ contains a filling disk with index
$m(r)\to\infty (r\to\infty).$ Since $J(f)$ is non-empty and
unbounded, $F(f)$ cannot contain any filling disks with large index
so that given arbitrarily $0<\sigma<1$, for large $r$, under
(\ref{equ1.7}), $F(f)$ cannot contain any annulus
$A(r^{1-\sigma},r)$.

Suppose that $f$ has a multiply connected Fatou component. Then
there exist a sequence of annuli $\{A(r_n^{1-\sigma},r_n)\}$ with
$r_n\to\infty$ for some $0<\sigma<1$ in the Fatou set $F(f)$; see
Theorem 1.2 in \cite{BRS} and Theorem 1.1 in \cite{zheng 16}. This
derives a contradiction. Theorem \ref{thm1.8} follows.

\section{Proof of Theorem \ref{thm1.3} and Theorem \ref{thm1.5+}}

\subsection{Some Lemmas}

We need the Nevanlinna characteristic in an angle; see \cite{gold
and ostr, zheng book 2}. We set
$$\Omega(\alpha,\beta)=\{z:\ \alpha<\arg z<\beta\}$$
with $0\leq\alpha<\beta\leq 2\pi$
and denote by $\overline{\Omega}(\alpha,\beta)$ the closure of
$\Omega(\alpha,\beta)$. Let $f(z)$ be meromorphic on the angle
$\overline{\Omega}(\alpha,\beta)$. We define
\begin{eqnarray*}A_{\alpha,\beta}(r,f)&=&\frac{\omega}{\pi}\int_1^r
\left(\frac{1}{t^\omega}-\frac{t^\omega}{r^{2\omega}}\right)
\{\log^+|f(te^{i\alpha})|+\log^+|f(te^{i\beta})|\}\frac{{\rm
d}t}{t};\\
B_{\alpha,\beta}(r,f)&=&\frac{2\omega}{\pi
r^w}\int_{\alpha}^{\beta}\log^+|f(re^{i\theta})|\sin
\omega(\theta-\alpha){\rm d}\theta;\\
C_{\alpha,\beta}(r,f)&=&2\sum_{1<|b_n|<r}\left(\frac{1}{|b_n|^\omega}-\frac{|b_n|^\omega}{r^{2\omega}}\right)\sin
\omega(\beta_n-\alpha),\end{eqnarray*} where
$\omega=\pi/(\beta-\alpha)$, and $b_n=|b_n|e^{i\beta_n}$ are poles
of $f(z)$ in $\overline{\Omega}(\alpha,\beta)$ appearing according
to their multiplicities and define
$C_{\alpha,\beta}(r,f=a)=C_{\alpha,\beta}(r,1/(f-a))$. The
Nevanlinna angular characteristic is defined as
$$S_{\alpha,\beta}(r,f)=A_{\alpha,\beta}(r,f)+B_{\alpha,\beta}(r,f)+C_{\alpha,\beta}(r,f).$$

\begin{lemma} (\cite{zheng book 2}, Lemma 2.2.2) \label{lemma 2}
Let $f(z)$ be a meromorphic function on $\overline{\Omega}(\alpha,
\beta)$. Then we have the following
$$C_{\alpha,\beta}(r,f=a)\leq 4\omega \frac{N(r,\Omega,f=a)}{r^{\omega}}
+2\omega^2 \int_{1}^{r}\frac{N(t,\Omega,f=a)}{t^{\omega+1}}{\rm
d}t.$$ The inequality also holds for $a=\infty$.
\end{lemma}

\begin{lemma} (\cite{zheng book 2}, The inequality (2.2.6) and Lemma 2.5.3) \label{lemma 3}
Let $f(z)$ be a meromorphic function. Then for any two distinct
values $a_1$ and $a_2$ on $\mathbb{C}$, we have
$$S_{\alpha,\beta}(r,f)\leq C_{\alpha,\beta}(r,f)+\sum_{\nu=1}^2 \overline{C}_{\alpha,\beta}(r,f=a_v)+O((\log rT(r,f)),$$
for all $r>0$ with a possible exception of finite-measure set of
$r$.
\end{lemma}

We need the following lemma, which is established in terms of the
hyperbolic metric.

\begin{lemma}\label{lem3.1} (\cite{zheng 16}, Theorem 2.4) \label{zheng's lemma}
Let $h(z)$ be an analytic function on the annulus $A(r,R)=\{z:
r<|z|<R\}$ with $0<r<R<\infty$ such that $|h(z)|>1$ on $A(r,R)$.
Then
\begin{equation}\label{equ3.1}
\log L(\rho,h) \geq
\exp\Big(-\frac{\pi^2}{2}\max\left\{\frac{1}{\log
\frac{R}{\rho}},\frac{1}{\log \frac{\rho}{r}}\right\}\Big)\log
M(\rho,h),
\end{equation}
where $\rho \in (r,R)$ and $L(\rho,h)=\min\{|h(z)|: |z|=\rho\}$.
\end{lemma}

In Lemma \ref{lem3.1}, when $\rho=\sqrt{rR}$, we have
\begin{equation}\label{equ}\log M(\rho,h) \leq
\exp\Big(\frac{\pi^2}{\log \frac{R}{r}}\Big)\log L(\rho,h).
\end{equation}

\subsection{Proof of Theorem \ref{thm1.3}}

On the contrary, suppose that there exist a sequence of annuli
$A_n=\{z:r_n<|z|<R_n\}$ in $F(f)$ with $R_n\geq(1+\phi(r_n)/\log
T(r_n,f))r_n$, $r_{n+1}>r_n$ and $r_n\to\infty (n\to \infty)$. We
treat two cases.

{\bf Case A.}\ There exists a subsequence of $A_n$ such that
$f(A_n)\subset\{|z|>1\}$. Without loss of generality, we assume that
for all $n$, $f(A_n)\subset\{|z|>1\}$. Set $\rho_n=\sqrt{R_nr_n}.$
Since $|f(z)|>1$ on $A_n$, using Lemma \ref{zheng's lemma}, we have
$$\log L(\rho_n,f)\geq \lambda_n\log M(\rho_n,f),$$
where $\lambda_n=\exp\left(-\frac{\pi^2}{\log
\frac{R_n}{r_n}}\right)$.

For the angular domain $\Omega(\alpha,\beta)$, according to the
definition of $B_{\alpha,\beta}(\rho_n,f)$, we have
\begin{equation} \label{4.1}
\begin{split} \rho_n^{\omega}B_{\alpha, \beta}(\rho_n,f)
      & = \frac{2\omega}{\pi}\int_{\alpha}^{\beta}
      \log^{+}|f(\rho_n e^{i\phi})|\sin (\omega(\phi-\alpha)){\rm d}\phi\\
      & \geq \frac{2\omega}{\pi}\log L(\rho_n,f)\int_{\alpha}^{\beta}
      \sin (\omega(\phi-\alpha)){\rm d}\phi\\
      & =\frac{4}{\pi}\log L(\rho_n,f)\geq\frac{4\lambda_n}{\pi}\log
      M(\rho_n,f).
\end{split}
\end{equation}

On the other hand, for any two distinct complex numbers $a_v
~(v=1,2)$, we use Lemma \ref{lemma 2} and Lemma \ref{lemma 3} in
turn to obtain that
\begin{eqnarray} \label{4.2}
\rho_n^{\omega}B_{\alpha, \beta}(\rho_n,f)&\leq &
\rho_n^{\omega}(C_{\alpha, \beta}(\rho_n,f=a_1)
+C_{\alpha, \beta}(\rho_n,f=a_2)) +O(\rho_n^{\omega}\log \rho_nT(\rho_n,f))\nonumber\\
      & \leq &4\omega N(\rho_n)+2\omega^2 \rho_n^{\omega} \int_1^{\rho_n}\frac{N(t)}{t^{\omega+1}}dt+O(\rho_n^{\omega}
      \log \rho_nT(\rho_n,f))\nonumber\\
      & \leq &4\omega N(\rho_n)+2\omega \rho_n^{\omega}N(\rho_n)+O(\rho_n^{\omega}\log \rho_nT(\rho_n,f))\nonumber\\
      & = &(4\omega +2\omega \rho_n^{\omega})N(\rho_n)+O(\rho_n^{\omega}\log \rho_nT(\rho_n,f)),
\end{eqnarray}
where $N(\rho_n)=N(\rho_n,\Omega,f=a_1)+N(\rho_n,\Omega,f=a_2)$.
Combining the inequalities \eqref{4.1}, \eqref{4.2} and (\ref{1})
yields $$(4\omega +2\omega
\rho_n^{\omega})N(\rho_n)+O(\rho_n^{\omega}\log \rho_nT(\rho_n,f))
\geq \lambda_n\log M(\rho_n,f))\geq K\lambda_nT(\rho_n,f),$$ for
some positive constant $K$. Thus
\begin{equation}\label{3.6}
K\lambda_nT(\rho_n,f)\leq 2\max\{(4\omega +2\omega
\rho_n^{\omega})N(\rho_n),O(\rho_n^{\omega}\log
\rho_nT(\rho_n,f))\}.\end{equation} It follows from the definition
of lower order that
$$\lim\limits_{n\to\infty} \frac{\log T(\rho_n,f)}{\log
\rho_n}\geq\lambda,$$ for the above sequence $\{\rho_n\}$. By noting
that
$$\log\lambda_n=-\frac{\pi^2}{\log(R_n/r_n)}\geq-\frac{\pi^2}{\log(1+\phi(r_n)/\log
T(r_n,f))}$$
$$\sim-\frac{\pi^2}{\phi(r_n)}\log T(r_n,f)\geq
-\frac{\pi^2}{\phi(r_n)}\log T(\rho_n,f),$$ we have
$$\lim\limits_{\overline{n\to\infty}}\frac{\log \lambda_n}{\log
T(\rho_n,f)}\geq 0.$$ In view of (\ref{3.6}), we can get
$$\lim_{n\rightarrow \infty}\frac{\log N(\rho_n)}{\log T(\rho_n,f)}\geq 1-\frac{\omega}{\lambda}.$$
Then we get a contradiction for $a$ and $b$.

{\bf Case B.}\ Set
$$c_n=1+\frac{\phi(r_n)}{12\log T(r_n,f)}.$$
When $c_n\leq 2$, we have $c_n^3\leq 1+\phi(r_n)/\log T(r_n,f).$
Consider the annulus $B_n=A(r_n,c_n^3r_n)\subset A_n$ and
$C_n=A(c_nr_n,c_n^2r_n)$. In view of implication in Case A, we can
assume that for all $n$, $f(C_n)\cap\{|z|\leq 1\}\not=\emptyset$.
Then there exist two points $z_0\in C_n$ and $z_n$ with
$|z_n|=\rho_n=c^{3/2}_nr_n$, in view of (\ref{1}),  such that
$$|f(z_0)|= 1+|c|\ \text{and}\ \log |f(z_n)|\geq KT(\rho_n,f)>\log\rho_n,\ {\text for}\ n\geq N,$$
where $c\in J(f)$ with $|c|=\min\limits_{z\in J(f)}|z|$, $K$ is a
positive number and $N$ is a positive integer. Thus for $n\geq N$
and $N\leq k\leq n$, $A_k\cap f(A_n)\not=\emptyset.$ This implies
that $\bigcup_{n=N}^\infty A_n\subset U$, where $U$ is a Fatou
component of $f$ with $f(U)\subseteq U$.

A simple calculation yields
$$\lambda_{B_n}(z)\leq\frac{2\sqrt{3}\pi}{9\log c_n}\frac{1}{\left|z\right|}, \
z\in C_n,\ \text{and}\ \lambda_{B_n}(z)=\frac{\pi}{3\log
c_n}\frac{1}{\left|z\right|},\ |z|=\rho_n.$$ Then
$$d_{B_n}(z_0,z_n)\leq \int_{c_nr_n}^{\rho_n}\frac{2\sqrt{3}\pi}{9\log c_n}\frac{{\rm
d}t}{t}+\int_0^\pi\frac{\pi}{3\log c_n}{\rm d}\theta \leq
\frac{\sqrt{3}\pi}{9}+\frac{\pi^2}{3\log c_n},$$ and
$$\frac{1}{\delta_n}\geq\frac{9\log c_n}{\sqrt{3}\pi\log
c_n+3\pi^2}\geq\frac{\log c_n}{4+\log c_n},$$ where
$\delta_n=d_{B_n}(z_0,z_n)$. It follows that for sufficiently large
$n$, we have
$$|f(z_n)|\geq \exp(KT(\rho_n,f))> e^{\kappa\delta_n}(1+2|c|)+|c|= e^{\kappa\delta_n}(|f(z_0)|+|c|)+|c|.$$
Set $g(z)=f(z)-c$. Then $0\not\in g(B_n)$ and $|g(z_n)|\geq
e^{\kappa\delta_n}|g(z_0)|$. In view of Lemma \ref{lem3.2}, we have
$g(B_n)\supseteq A(d^{-1}_nt_n, d_nt_n)$ with $t_n\geq |g(z_0)|\geq
1$ and
$$d_n:=e^{-\kappa}\left(\frac{|g(z_n)|}{|g(z_0)|}\right)^{1/\delta_n}\geq
\exp\left(-\kappa+\frac{\log c_n}{4+\log
c_n}\frac{K}{2}T(\rho_n,f)\right)>\rho_n^5.$$ Thus
$$f(B_n)\supseteq A(d^{-1}_nt_n, d_nt_n)+q\supseteq A(d^{-1}_nt_n+|c|, d_nt_n-|c|).$$
Set $s_n=\rho_nt_n\geq \rho_n$. Then
$$A(s_n, \rho_n^3s_n)\subset A(d^{-1}_nt_n+|c|, d_nt_n-|c|).$$
This implies that the Fatou component $U$ contains a sequence of
annuli $D_n=A(s_n, \rho_n^3s_n)$.

For simple statement, we assume without loss of generality that
$c=0$. We can assume that $f(E_n)\cap\{z:\ |z|\leq
1\}\not=\emptyset$ with $E_n=A(\rho_ns_n, \rho_n^2s_n)$. We can find
two points $z_0'\in E_n$ and $z_n'$ with $|z_n'|=\rho_n^{3/2}s_n$
such that
$$|f(z_0')|=1\ \text{and}\ \log |f(z_n')|\geq KT(\rho_n^{3/2}s_n,f)>2\log r_n.$$
And we have
$$d_{D_n}(z_0',z_n')\leq\frac{\sqrt{3}\pi}{9}+\frac{\pi^2}{3\log \rho_n}<\frac{9}{10}.$$
Therefore, in view of Lemma \ref{lem3.2}, $f(D_n)\supset
A(|f(z_0')|,f(z_n')|)\supset A(1, r_n^2)$ and furthermore,
$A(1,r_n^2)\subset U$ and so $A(1,R_n)\subset U$. This implies that
$U\supset \{z:\ |z|>1\}$. A contradiction is derived.

From Cases A and B, the Fatou set $F(f)$ contain no annuli mentioned
in Theorem \ref{thm1.3}.

\subsection{Proof of Theorem \ref{thm1.5+}}

Consider the meromorphic function with the following form
\begin{equation}f(z)=cz+d+\sum_{n=1}^\infty
c_n\left(\frac{1}{a_n-z}-\frac{1}{a_n}\right),
\end{equation}
where $c, d, c_n$ and $a_n$ are real numbers with $a_n\to \infty
(n\to\infty)$, $c\geq 0$ and $c_n>0$ such that
$$\sum_{n=1}^\infty\frac{c_n}{a_n^2}<+\infty.$$
For such a function $f$, the real axis is completely invariant under
$f$ and so $J(f)$ is completely on the real axis, see \cite{BKY}.

Set $r_n=M2^n$ for a large $M>0$. For a given real number
$\lambda>\frac{1}{2}$, define
$r_{n,k}=r_{n+1}-1+\frac{k}{[r_n^\lambda]},\ 1\leq k\leq
[r_n^\lambda]$ for each $n$, where $[x]$ is the maximal integer not
greater than $x$. We will prove that the function
$$g(z)=\sum_{n=1}^\infty\frac{1}{[r_n^\lambda]}\sum_{k=1}^{[r_n^\lambda]}\frac{2z}{r_{n,k}^2-z^2}
=\sum^\infty_{n=1}\frac{1}{[r_n^\lambda]}\sum_{k=1}^{[r_n^\lambda]}\left(\frac{1}{r_{n,k}-z}-\frac{1}{r_{n,k}+z}\right)$$
satisfies the requirement of Theorem \ref{thm1.5+}.

Given arbitrarily a large real number $r$, we have $r_n\leq
r<r_{n+1}$ for some $n$. Then
\begin{eqnarray*}n(r,f)&=&\sum_{k=1}^{n-1}[r_k^\lambda]
\leq r_n^\lambda\sum_{k=1}^{n-1}\left(\frac{r_k}{r_n}\right)^\lambda\\
&=&r^\lambda\sum_{k=1}^{n-1}\left(\frac{1}{2^{n-k}}\right)^\lambda\\
&\leq & \frac{r^\lambda}{2^\lambda-1}
\end{eqnarray*}
and
\begin{eqnarray*}n(r,f)&\geq&\sum_{k=1}^{n-1}(r_k^\lambda-1)=\frac{r_n^\lambda}{2^\lambda-1}-n\\
&=&\frac{r_{n+1}^\lambda}{2^\lambda(2^\lambda-1)}-\frac{\log(r_n/M)}{\log
2}\\
&>&\frac{r^\lambda}{2^\lambda(2^\lambda-1)}-\frac{\log(r/M)}{\log
2}.\end{eqnarray*} For $z\in A(\frac{4}{3}r_n,\frac{5}{3}r_n)$, we
have
$$|g(z)|\leq\sum_{m=1}^{n-1}\frac{1}{[r_m^\lambda]}\sum_{k=1}^{[r_m^\lambda]}\frac{2|z|}{|z|^2-r_{m,k}^2}+
\sum_{m=n}^{\infty}\frac{1}{[r_m^\lambda]}\sum_{k=1}^{[r_m^\lambda]}\frac{2|z|}{r_{m,k}^2-|z|^2}$$
$$=I_1+I_2(say).$$
We estimate
\begin{eqnarray*}I_1&\leq &\sum_{m=1}^{n-1}\frac{1}{[r_m^\lambda]}\sum_{k=1}^{[r_m^\lambda]}
\frac{30r_n}{16r_n^2-9r_{m,k}^2}\\
&\leq&\frac{30nr_n}{16r_n^2-9r_n^2}=\frac{30n}{7r_n}.
\end{eqnarray*}
Since $r_{m,k}\geq r_{m+1}-1=2r_m-1$, for $m\geq n$ we have
$$9r^2_{m,k}-25r_n^2\geq 9(2r_m-1)^2-25r_n^2>10r_m^2$$ so that
\begin{eqnarray*}I_2&\leq&\sum_{m=n}^{\infty}\frac{1}{[r_m^\lambda]}\sum_{k=1}^{[r_m^\lambda]}
\frac{30r_n}{9r_{m,k}^2-25r_n^2}\\
&<&\sum_{m=n}^{\infty}\frac{30r_n}{10r_{m}^2}\leq
3r_n\sum_{m=n}^{\infty}r_m^{-2}=\frac{3}{r_n}\sum_{m=n}^{\infty}\left(\frac{r_n}{r_m}\right)^{2}<\frac{3}{r_n}.
\end{eqnarray*}
Therefore we have
$$T(3r_n/2,g)=N(3r_n/2,g)+m(3r_n/2,g)<n(3r_n/2,g)\log(3r_n/2)+2$$
and
$$T(3r_n/2,g)=N(3r_n/2,g)+m(3r_n/2,g)>n(r_n,g)\log(3/2)-2,$$
and so
$$\lim_{n\to\infty}\frac{\log T(3r_n/2,g)}{\log (3r_n/2)}=\lambda.$$
This easily implies that $g$ has the order and lower order equal to
$\lambda$.

Obviously, $0$ is an attracting fixed point of $g$ and we can choose
a large $M$ such that $B(0,M)$ is in its attracting basin. For all
sufficiently large $n$, the annulus $A(4r_n/3,5r_n/3)$ is mapped
into $B(0,M)$. Therefore the Fatou set $F(g)$ contains a sequence of
annuli $A(4r_n/3,5r_n/3)$ with $r_n\to\infty (n\to\infty)$. Since
$J(g)$ lies in the real axis and $r_1,r_2\in J(g)$, $g$ cannot take
on $r_1$ and $r_2$ on the upper half plane and lower half plane.

\section{Remarks}

In the proofs of Theorem \ref{thmA} and Theorem \ref{th2}, the
existence of so-called filling disks is a key point.

From the proof of Theorem \ref{thmA} in \cite{qiao}, it is easily
seen that if $0<\lambda(f)<\infty$, then there exist $R_0, \tau>1$
and $N>0$ such that for all $r>R_0$, the annulus $A(r/4, 3\tau
T^{-1}(T^N(r)))$ contains a filling disk $D:\
|z-z_0|<\frac{4\pi}{\log r}|z_0|$ of $f$ with the index
$m=c^*\frac{T(R)}{(\log R)^2},$ where $c^*$ is a positive constant,
$T(r)=T(r,f)$ and $\tau r<R<\tau T^{-1}(T^N(r))$.

From the proof of Theorem \ref{th2}, we see that the annulus $A(r,
2r^2)$ for $r\geq R_0$ contains a filling disk $\Gamma:
|z-z_0|<\frac{4\pi}{\log r}|z_0|$ with index
$m=c^*\frac{T(r^2)}{(\log r)^4}.$ However, for any two sequences
$\{r_n\}$ and $\{R_n\}$ of positive numbers with $r_n<R_n<r_{n+1}$,
there exists a transcendental meromorphic function which has no
filling disks in annulus $A(r_n, R_n)$ with index $m_n\to\infty
(n\to\infty)$. For such example which has a sequence of large annuli
in the Fatou set, see \cite{ZX}. A transcendental entire function
$f$ with this property can be found in \cite{BZ}. It is obvious that
the Fatou set cannot contain a filling disk with enough large index
$m$. This is a contradiction of the iterate theory researches of a
meromorphic function to the study of value distribution.

\bigskip
\noindent \textbf{Acknowledgements.} This paper was supported by
Grant of 11571193 of the NSF of China.


\begin{thebibliography}{99}

\bibitem{baker}
I. N. Baker,
\emph{The domain of nromality of an entire function}, Ann. Acad. Sci. Fenn. Ser. AI
Math., \textbf{1} (1975), 277-283.

\bibitem{BKY}
I. N. Baker, J. Kotus and Y. L\"u, \emph{Iterates of meromorphic
functions I, Ergodic Theory and Dynamics System, \textbf{11} (1991),
241-248}

\bibitem{beardon}
A.~F.~Beardon,
\emph{Iteration of Rational Functions}.
Spinger, Berlin, 1991.


\bibitem{bergweiler}
W.~Bergweiler,
\emph{Iteration of meromorphic functions},
Bull. Amer. Math. Soc. (N. S.) \textbf{29} (1993), 151-188.

\bibitem{BH}  W. Bergweiler and A. Hinkenna, \textit{On
semiconjugation of entire functions,} Math. Proc. Cambridge Philos.
Soc., 126(1999), 565-574

\bibitem{BRS}
W.~Bergweiler, P. ~Rippon and G. ~Stallard, \textit{Multiply
connected wandering domains of entire functions}, Proc. London Math.
Soc., (3) 107 (2013), 1261-1301


\bibitem{BZ}
W.~Bergweiler and J. H. Zheng, On the uniform perfectness of the
boundary of multiply connected wandering domains, J. Aust. Math.
Soc., 91(2011),289-311.

\bibitem{DK}
R. L. Devaney and M. Krych, Dynamics of $\exp (z)$, Ergod. Th. \&
Dynam. Sys., (1984), 4, 35-52

\bibitem{DT}
R. L. Devaney and F. Tangerman, Dynamics of entire functions near
the essential singularity, Ergod. Th. \& Dynam. Sys., (1986), 6,
489-503

\bibitem{Dominguez} P. Dominguez, \textit{Dynamics of transcendental
meromorphic functions}, Ann. Acad. Sci. Fenn. Math., 23(1998),
225-250

\bibitem{eremenko}
A.~E.~Eremenko,
\emph{On the iteration of entire functions},
Dynamical Systems and Ergodic Theory (Banach
Center Publications, 23). Polish Scientific Publishers, Warsaw, (1989), 339-345.

\bibitem{gold and ostr}
A.~A.~Goldberg and I.~V.~Ostrovskii,
\emph{Value Distribution of Meromorphic Functions}.
Transl. Math. Monogr. 236, Amer. Math. Soc., 2008.

\bibitem{Hayman}
W.~K.~Hayman,
\emph{Meromorphic Functions}.
Clarendon Press, Oxford, UK, 1964.

\bibitem{milnor}
J.~Milnor,
\emph{Dynamics in One Complex Variable}.
Friedr. Vieweg $\&$ Sohn, Braunschweig, 1999.

\bibitem{Petrenko}V. P. Petrenko, On the growth of meromorphic
functions with finite order (in Russian), Isv. Akad. Nauk. SSSR,
33(1969), 414-454

\bibitem{qiao}
J. Y. Qiao, \emph{Borel directions and iterated orbits of
meromorphic functions}, Bull. Austral. Math. Soc. \textbf{61(1)}
(2000), 1-9.

\bibitem{R06}
L. Rempe, Rigidity of escaping dynamics for transcendental entire
functions, Acta Math.,

\bibitem{R07}
L. Rempe, On a question of Eremenko concerning escaping components
of entire functions, Bull. London Math. Soc., 39(2007), 661-666

\bibitem{RRS}
L. Rempe, P. J. Rippon and G. M. Stallard, Are Devaney hairs fast
escaping? J. Differences and Appl., (2010), 6(5-6), 739-762

\bibitem{RS} P.J. Rippon and G.M. Stallard, \textit{On
sets where iterates of a meromorphic function zip towards infinity},
Bull. London Math. So., 32(2000), 528-536

\bibitem{RipponStallard} P.J. Rippon and G.M. Stallard, \textit{Slow escaping points of
meromorphic functions,} Trans. Amer. Math. Soc., {\bf 363}(2011),
4171-4201

\bibitem{RS1} P.J. Rippon and G.M. Stallard,
\textit{Fast escaping points of entire functions,} Proc. London
Math. Soc., 105(2012), 787-820

\bibitem{RS19} P.J. Rippon and G.M. Stallard, Eremenko points and
the structure of the escaping set, Trans. Amer. Math. Soc.,
372(2019), no. 5, 3083-3111

\bibitem{Six} D. J. Sixsmith, Maximally and non-maximally fast
escaping points of transcendental entire functions, Math. Proc.
Camb. Phil. Soc. (2015), 158, 365-383

\bibitem{Vi} M. Viana da Silva, The differentiability of the hairs
of $\exp (z)$, Proc. Amer. Math. Soc., 103(4) (1988), 1179-1184


\bibitem{yang book}
L. Yang,
\emph{Value distribution theory}, Springer-Verlag, Berlin, Heidelberg, New York; Science
Press, Beijing, 1993.

\bibitem{zheng book}
J.~H.~Zheng,
\emph{Dynamics of Transcendental Meromorphic Functions}.
Tsinghua University Press, Beijing, 2006. (In Chinese)

\bibitem{zheng}
J.~H.~Zheng,
\emph{On multiply-connected Fatou components in iteration of meromorphic functions},
J. Math. Anal. Appl. \textbf{313} (2006), 24-37.

\bibitem{zheng 16}
J. H. Zheng, \emph{Hyperbolic metric and multiply connected
wandering domains of meromorphic functions}, Proc. Edinburg Math.
Soc., \textbf{60} (2017), 787-810.

\bibitem{zheng book 2}
J. H. Zheng, \emph{Value Distribution of Meromorphic Functions}, Tsinghua Univesity Press Beijing
and Springer-Verlag Berlin, 2010.

\bibitem{Zheng3} J. H. Zheng, \textit{Domain Contants and their
applications in dynamics of meromorphic functions}, J. Jiangxi
Normal University (Natural Science), vol. 34, no. 5(2010), 1-7

\bibitem{ZX} J. H. Zheng, \textit{Fast Escaping Sets and Existence of Large Annulus of meromorphic functions}


\end{thebibliography}
\end{document}